\documentclass{amsart}
\usepackage{amscd, amssymb, mathrsfs, mathabx, 
dsfont, xcolor, amsmath, soul, color, mathtools}

\usepackage{hyperref}  

\usepackage{collectbox}
\usepackage{enumitem}



\input{xy}
\xyoption{all}

\newtheorem*{theorem*}{Main Theorem}

\newtheorem{theorem}{Theorem}[section]
\newtheorem{proposition}[equation]{Proposition}
\newtheorem{corollary}[theorem]{Corollary}
\newtheorem{lemma}[theorem]{Lemma}

\theoremstyle{definition}
\newtheorem{definition}{Definition}[section]

\theoremstyle{remark}
\newtheorem{Rmk}[theorem]{Remark}

\numberwithin{equation}{section}

\newcommand{\Ker}{{\mathrm{Ker}}}
\newcommand{\IIm}{{\mathrm{Im}}}

\newcommand{\A}{{\mathbb A}}

\newcommand{\kb }{{\beta}}
\newcommand{\ka }{{\alpha}}

\newcommand{\ch}{\mathrm{ch}}

\newcommand{\Chhn}{\mathrm{Ch}^{G}_{\mathrm{HN}}}
\newcommand{\Chhp}{\mathrm{Ch}^{G}_{\mathrm{HP}}}

\newcommand{\chg}{\mathrm{ch}^{G}}
\newcommand{\chxh}{\mathrm{ch}^{G}_{X_{\infty}^{g}}}
\newcommand{\chhh}{\mathrm{ch}^{G}_{\mathrm{HH}}}
\newcommand{\chpv}{\mathrm{ch}^{G}_{\mathrm{PV}}}
\newcommand{\chhn}{\mathrm{ch}^{G}_{\mathrm{HN}}}
\newcommand{\chhp}{\mathrm{ch}^{G}_{\mathrm{HP}}}

\newcommand{\fm}{\mathfrak{m}}
\newcommand{\MF}{\mathrm{MF}}
\newcommand{\id}{\operatorname{id}}
\newcommand{\Homb}{\mathcal{H}om} 
\newcommand{\End}{\mathrm{End}}

\newcommand{\sTr}{\operatorname{str}}

\newcommand{\Spec}{{\mathrm{Spec}}}
\newcommand{\Q}{\mathbb{Q}}
\newcommand{\JJ}{\mathcal{J}}

\newcommand{\cG}{{\mathcal{G}}}

\def\on{\operatorname}
\def\idem{\on{idem}}
\newcommand{\sK} {\mathscr{K}}
\def\an{{\mathrm{an}}}
\def\Gres{\on{Res}_{f_g}}
\def\GGres{\on{Res}_{f}}
\def\rGres{\on{res}}
\def\rrGres{\on{Res}_{f_h}}

\newcommand{\sP} {\mathscr{P}_S}
\newcommand{\tr }{{\mathrm{tr}}}
 \newcommand{\gtr }{\tr_{(Q^\hen)^g}}

\newcommand{\hen}{\mathfrak{h}}

\newcommand{\Var}{{\mathsf{Var}}}

\newcommand{\paa}{\partial}

\newcommand{\CC }{{\mathbb C}}
\newcommand{\ot }{\otimes}
\newcommand{\cO}{{\mathcal{O}}}
\newcommand{\cy}{{\mathsf{cy}}}
\newcommand{\QQ }{{\mathbb Q}}
\newcommand{\ttop}{{\mathsf{top}}}
\newcommand{\ccan}{{\mathsf{can}}}

\def\into{\hookrightarrow}

\begin{document}
\title{Standard conjecture D for local stacky matrix factorizations}

\author[B. Kim]{Bumsig Kim}
\address{Korea Institute for Advanced Study\\
85 Hoegiro, Dongdaemun-gu \\
Seoul 02455\\
Republic of Korea}
\email{bumsig@kias.re.kr}

\author[T. Kim]{Taejung Kim }
\address{Department of Mathematics Education\\
  Korea National University of Education\\
250 Taeseongtabyeon-ro, Gangnae-myeon\\
Heungdeok-gu, Cheongju-si, Chungbuk 28173\\
 Republic of Korea}
\email{tjkim@kias.re.kr}


\thanks{B. Kim was supported by KIAS individual grant MG016404 and T. Kim was supported by NRF-2018R1D1A3B07043346.}

\begin{abstract}
We establish the non-commutative analogue of Grothendieck's standard conjecture D for the differential graded category of $G$-equivariant matrix factorizations associated to an isolated hypersurface singularity where $G$ is a finite group.
\end{abstract}


\maketitle


\section{Introduction}

Let $X$ be a smooth projective variety over a field $k$. The standard conjecture D formulated by Grothendieck states that the numerical equivalence of algebraic cycles of $X$ coincides with the homological equivalence; consider the following diagram
\[ \xymatrix{ A^*(X)\otimes_{\mathbb{Z}}\QQ \ar[r]^(.5){\cy} &  H^*(X ; \QQ) \\
                   K_0(X)\otimes_{\mathbb{Z}}\QQ \ar[ru]_{\ch ^{\ttop}} \ar[u]^{\ch }_{\cong} &  } \] 
                    where  $A^*(X)$ is the Chow ring,  $K_0(X)$ is the Grothendieck group of coherent sheaves on $X$, $H^*(X ; \QQ)$ is singular cohomology, $\cy$ is an associated cycle class map, $\ch$ is the Chern character map (see \cite[Chapter 15]{fulton} for details), and $\ch^\ttop$ is the topological Chern character map.
By the isomorphism $\ch$ in the diagram and the Hirzebruch–Riemann–Roch theorem, the standard conjecture D can be restated;  for $\ka \in K_0(X)\otimes_{\mathbb{Z}}\QQ:=K_0 (X)_\QQ$ the Euler characteristic pairing $\chi (\kb, \ka ) = 0$ for all $\beta \in K_0 (X)_\QQ$ if and only if $\ch^{\ttop} (\ka ) = 0$. To date this conjecture remains essentially open besides cases under some conditions, for instance, when $X$ is a complete intersection, $\dim_\mathbb{C} X\leq 4$, or $X$ an abelian variety.

After formulation of a non-commutative generalization of the conjecture by Marcolli and Tabuada \cite{MT, tabuada2}, Brown and Walker \cite{BW2} recently prove that the standard conjecture D holds for the dg-category of matrix factorizations of an isolated hypersurface singularity in characteristic $0$. Our aim in this paper is to extend their result to the case of $G$-equivariant matrix factorizations where $G$ is a finite group;

\begin{theorem*}
  Under \ref{enun1} and \ref{enun2} in \S~\ref{mil_fib}, for $\ka \in K_0(\MF_G(Q,f))$ the Euler pairing $\chi(\kb, \ka)= 0$ for all $G$-equivariant matrix factorizations $\kb\in K_0(\MF_G(Q,f))$ if and only if $\chhp(\ka) = 0$ where $\chhp$ is the $G$-equivariant periodic cyclic Chern character.
\end{theorem*}

\subsection{Outline of paper}
In \S~\ref{mil_fib}, after recalling definitions and some preparatory material about the Milnor fibration; see \cite{Hertlingbook, Kulikov}, we endow the vanishing cohomology with an explicit polarized mixed Hodge structure in the sense of Scherk and Steenbrink \cite{SS, steenbrink} and Hertling \cite{Hertling} by realizing it as a subspace of a certain $\mathcal{D}$-module and the positive  definite subspace of a polarization $S$ is described in \S~\ref{PMHS}.  In \S~\ref{HRpair} we introduce two higher residue pairings, K. Saito's higher residue pairing and a higher residue pairing associated with the polarization $S$. Their relations are characterized in Proposition~\ref{prop_1} and make it possible to associate the polarization $S$ to Grothendieck's residue pairing. In \S~\ref{milequi}, \S~\ref{gmat}, and \S~\ref{grogp}, we collect definitions and fundamental facts of the Milnor fibration in the $G$-equivariant setting, the $G$-equivariant matrix factorizations, and the Grothendieck group of the category.
Moreover, in \S~\ref{chench} and \S~\ref{chenhn} we realize the Hochschild homology and the negative cyclic homology in explicit ways via the Hochschild-Kostant-Rosenberg isomorphisms. It enables us to characterize two $G$-equivariant Chern characters explicitly in \S~\ref{mainproof}. Adapting the strategy of proving the non-equivariant case by Brown and Walker \cite{BW2} to the equivariant point of view, in \S~\ref{chentype} we introduce the $G$-equivariant version $\chxh$ of Chern-character-type map $\ch_{X_\infty}$ and in \S~\ref{mainproof}, we prove the main theorem, i.e., Theorem~\ref{math1}, by relating the Euler pairing $\chi$ with the polarization $S$ by applying the $G$-equivariant version of the Hirzebruch–Riemann–Roch theorem for matrix factorizations by Polishchuk and Vaintrob \cite{PV: HRR}.

\section*{Acknowledgments}
The second author would like to express his deepest gratitude to the first author, Bumsig Kim who passed away prior to the completion of this manuscript, for sharing his friendship and mathematical insights.

We thank the anonymous referee for careful reading of the manuscript and the helpful comments and suggestions.

\section{The Milnor fibration and the polarized mixed Hodge structure}
\subsection{Milnor fibration and Brieskorn lattice}\label{mil_fib}
Assume that 
  \begin{enumerate}[label=(\roman*)]
\item $Q = \mathbb{C}[x_0, \dots, x_n]$;\label{enun1}
\item $f$ is an element of $\fm := (x_0, \dots, x_n) \cdot Q$ 
such that  the only singularity of the associated morphism $f: \Spec(Q) \to \A^1_\mathbb{C}$  of smooth affine varieties
is at $\fm$.\label{enun2}
\end{enumerate}
  Consider an open ball $B_{\epsilon}$ of radius $\epsilon$ small enough centered at the origin in $\mathbb{C}^{n+1}$. Letting $X:= f^{-1}(T) \cap B_{\epsilon}$, $f: X \to T$ denotes the map induced by $f: \Spec(Q) \to \A^1_\mathbb{C}$ by abuse of notation where $T$ is an open disc of radius $\eta$ centered at the origin in $\mathbb{C}^1$ with $\eta\ll\epsilon$ so that $f' : X' \to T'$ becomes a fibration where $X' = X\setminus f^{-1}(0)$ and $T'=T\setminus\{0\}$.

Let $\pi:T_\infty \to T'$ be the universal cover of $T'$ and $X_\infty:=X'\times_{T'}T_\infty$, which is homotopy equivalent to $\bigvee^{\mu} S^n$ by a theorem of Milnor where $\mu$ is called the Milnor number. Since $f'$ is a fibration over $T'$,  $H^n(X_\infty ;  \mathbb{C})$, the so-called vanishing cohomology, is equipped with a monodromy operator $M$. Its generalized eigenspace corresponding to an eigenvalue $\lambda$  is denoted by 
$$
H^n(X_\infty; \mathbb{C})_\lambda:= \bigcup_{j \geq 1}  \ker((M-\lambda\cdot \id)^j),$$
$H^n(X_\infty; \mathbb{Q})_{\neq 1}:= H^n(X_\infty; \mathbb{Q})\bigcap \bigoplus_{\lambda\ne1}H^n(X_\infty; \mathbb{C})_\lambda$, and similarly $H^n(X_\infty; \mathbb{Q})_\lambda$ is defined.

Let $\mathcal{G}:=i_\ast(\mathbf{R}^n f'_* \mathbb{C}_{X'}  \otimes_{\mathbb{C}_{T'}} \cO^{an}_{T'})$ where  $i: T' \into T$ is the inclusion and  $\mathcal{G}_0$  be the stalk of $\mathcal{G}$ at the origin, which is a vector space over $\mathbb{C}\{t\}[t^{-1}]$. Choosing a value of logarithm $\alpha=(-1/2\pi i)\ln \lambda$, one may define a mapping
\begin{equation}\label{psi1}
  \psi_\alpha:H^n(X_\infty; \mathbb{C})_\lambda\to\mathcal{G}_0\text{ by }
  \psi_\alpha(A):=es(A,\alpha)_0
  \end{equation}
where the elementary section $es(A,\alpha)(t):=t^\alpha\exp(-\log t\cdot \frac{N}{2\pi i})A(t)$ and $A(t)$ is the identification of $A\in H^n(X_\infty;\mathbb{C})$ with the space of the global flat multi-valued sections of the cohomology bundle $\bigcup_{t\in T'}H^n(X_t;\mathbb{C})$ where $X_t:=f^{-1}(t)$; see \cite[Section 9]{Kulikov} and \cite[Section 7.1]{Hertlingbook}. Note that $\psi_\alpha$ is injective and the image $\psi_\alpha(H^n(X_\infty;\mathbb{C})_\lambda)$ is $C_\alpha:=\ker(t\partial_t-\alpha)^{n+1}\subset\mathcal{G}_0$, i.e., $\bigoplus_{-1<\alpha\leq0}\psi_\ka:H^n(X_\infty;\CC)\to\bigoplus_{-1<\alpha\leq0}C_\ka$ is an isomorphism (see \cite{Hertling,Hertlingbook, Kulikov}) where the action of $\partial_t$ on $\mathcal{G}_0$ is given by $\partial_t t=\id+t\partial_t$; see \cite[Section 4.2]{BW2} and \cite[Remark 4.8]{BW2} for a complete description of $\partial_t$. One can also see that $t: C_\ka \to C_{\ka + 1}$ is an isomorphism  for each $\ka$ and
$\partial_t: C_{\ka+1} \to C_\ka$ is an isomorphism for all $\ka \ne -1$; see \cite[Section 6]{Kulikov} and \cite[Section 4.3]{BW2}.

Let $\omega\in \Omega^{\an, n+1}_{X,0}$ and 
\begin{equation}\label{s_0}
  s_0:H''_{0}\to\mathcal{G}_0\text{ by }s_0([\omega]_B) := \left( t \mapsto \left[ \frac{\omega}{df}|_{X_t}\right] \in H^n(X_t; \mathbb{C}) :  0 < |t| \ll 1\right)_0
  \end{equation}
where $[\omega]_B$ is the class in a Brieskorn lattice $H''_{0}:=\Omega^{\an, n+1}_{X,0}/df \wedge d \Omega^{\an, n-1}_{X,0}$ and the $\frac{\omega}{df}$ is the Gelfand-Leray form of $\omega$ (see \cite[Section 7.1]{AGZV}). Let $V^{>\alpha}:=\sum_{\beta>\alpha}\mathbb{C}\{t\}C_\kb$. We note that $s_0$ is a well-defined injective map, the image of $s_0: H_0'' \to \mathcal{G}_0$ is contained in $V^{>-1}$, and the map $s_0: H_0'' \to V^{>-1}$ is $\CC\{t\}\langle \partial_t^{-1} \rangle$-linear; see \cite[Lemma 4.7]{BW2}.

\subsection{Polarized mixed Hodge structure on $H^n(X_\infty; \QQ)$ }\label{PMHS}
We describe a polarized mixed Hodge structure of level $n$ and level $n+1$ on $H^n(X_\infty; \mathbb{Z})$ as follows; see \cite{Hertling} and
\cite[Remarks 10.25]{Hertlingbook}. Take $N=-\log M_u$ where $M_u$ is the unipotent part of the monodromy operator $M$. It is an endomorphism of $H^n(X_\infty; \Q)$ such that $N^{n+2} = 0$ by the monodromy theorem \cite{Hertling,Kulikov}. The filtration $F^\bullet$ on $H^n(X_\infty; \mathbb{C})_\lambda$ is given by
 \begin{equation}\label{pmhsst}
   F^q H^n(X_\infty; \mathbb{C})_\lambda := \psi^{-1}_\alpha \left( \frac{V^\alpha \cap \partial_t^{n-q} s_0(H_0'') + V^{>\alpha}}{V^{>\alpha}}  \right)\text{ where }V^{\alpha}:=\sum_{\beta\geq\alpha}\mathbb{C}\{t\}C_\kb.
\end{equation}

For the Milnor fibration $f:X\to T$, the intersection form and the long exact sequence associated with relative homology groups induce a canonical isomorphism
\begin{equation}
H^n(X_t;\mathbb{Z})\cong H_n(X_t,\paa X_t;\mathbb{Z})\text{ for }t\ne0.
\end{equation}
A variation operator, which is an isomorphism \cite{Hertling, He3}, is defined by
\begin{equation}
  \Var: H^n(X_t;\mathbb{Z})\to H_n(X_t;\mathbb{Z})\text{ by }
  \Var([\gamma]):= [M(\gamma)-\gamma]
\end{equation}
where $\gamma$ is a representative of the cycle 
$[\gamma]\in H_n(X_t,\partial X_t;\mathbb{Z})$.
Noting that $X_\infty$ is homotopy equivalent to $X_t$ for $t\ne0$, we may identify $H_n(X_\infty;\mathbb{Z})$ with $H_n(X_t;\mathbb{Z})$.
It induces a variation operator
\begin{equation}
  \Var: H^n(X_\infty;\mathbb{Z})\to H_n(X_\infty;\mathbb{Z}).
\end{equation}
Hertling's polarization $S: H^n(X_\infty;\mathbb{Q})\times H^n(X_\infty;\mathbb{Q})\to \mathbb{Q}$ is defined by
\begin{equation}\label{2.36}
S(a,b): = (-1)^{n(n-1)/2}\cdot \langle a, \Var\circ \nu(b)\rangle\text{ where }
\end{equation}
\begin{align}
\nu &= (M-\id)^{-1}  &\hbox{ on } 
H^n(X_\infty;\mathbb{Q})_{\neq 1};&\text{ level }n,\\
\nu &= \sum_{l\geq 1}\frac{1}{l}(-1)^{l}(M-\id)^{l-1}:=\frac{-N}{M-\id} &\hbox{ on } H^n(X_\infty;\mathbb{Q})_{1};&\text{ level }n+1
\end{align}
where $\langle-,-\rangle$ denotes the canonical bilinear form between
$H^n(X_\infty;\mathbb{Q})$ and $H_n(X_\infty;\mathbb{Q})$. It is known that $S$ is a polarization on the mixed Hodge structure on $H^n(X_\infty;\mathbb{Q})$ (see \cite[Sections 3 and 4]{Hertling} for details).
\begin{Rmk}\label{remark1}
  As a generalization of a relation between the primitive cohomology and the intersection pairing in classical Hodge theory,
  one has the following property in a polarized mixed Hodge structure; let $m$ be an even integer and suppose $(H,N,F^\bullet,S)$ is a polarized mixed Hodge structure of level $m$. The restriction of the polarization $S$ to a subspace $(\Ker N : H_\mathbb{C}\to H_\mathbb{C})\cap F^{\frac{m}{2}}H_\mathbb{C}$ is positive definite (see \cite[Lemma A.6]{BW2}).
\end{Rmk}

\subsection{Higher residue pairings}\label{HRpair}

\subsubsection{K. Saito's higher residue pairing}\label{Saitono}
On a Brieskorn lattice $H''_{0}:=\Omega^{\an, n+1}_{X,0}/df \wedge d \Omega^{\an, n-1}_{X,0}$, there is, so-called, Saito's higher residue pairing (see \cite{Hertlingbook,He3} for details) 
$$\begin{aligned}
  \sK&:H''_{0}\times H''_{0}\to \mathbb{C}[\![\partial_{t}^{-1}]\!]\partial_{t}^{-n-1}\\
  \sK(\omega_1,\omega_2)&:=\sum_{m\geq0}\sK^{(-m)}(\omega_1,\omega_2)\partial_{t}^{-n-1-m}\text{ such that}
\end{aligned}$$
\begin{enumerate}[label=\arabic*)]
\item $\sK^{(-m)}(\omega_1,\omega_2)\in\mathbb{C}$ is $\mathbb{C}$-linear and $(-1)^m$-symmetric;
  \item $\sK^{(0)}(\partial_{t}^{-1}\omega_1,\omega_2)=\sK^{(0)}(\omega_1,\partial_{t}^{-1}\omega_2)=0$ and
$$\sK^{(-m)}(\omega_1,\omega_2)=\sK^{(-m-1)}(\partial_{t}^{-1}\omega_1,\omega_2)=-\sK^{(-m-1)}(\omega_1,\partial_{t}^{-1}\omega_2);$$
    \item   
      $$\sK^{(0)}([\psi_1dx_0\cdots dx_{n}]_B,[\psi_2dx_0\cdots dx_{n}]_B)
      :=\rGres\left[  \frac{\psi_1\psi_2dx_0 \cdots dx_{n}} {f_0, \dots, f_{n}} \right]$$
    where $f_i:=\frac{\partial f}{\partial x_i}$ and $\rGres$ is Grothendieck's residue;
\item $(m-1)\sK^{(-m+1)}(\omega_1,\omega_2)=\sK^{(-m)}(t\omega_1,\omega_2)-\sK^{(-m)}(\omega_1,t\omega_2)$.
\end{enumerate}
We will denote $\rGres\left[  \frac{\psi_1\psi_2dx_0 \cdots dx_{n}} {f_0, \dots, f_{n}} \right]$ by $\GGres(\psi_1dx_0\cdots dx_{n},\psi_2dx_0\cdots dx_{n})$ in \S~\ref{mainproof}.

\subsubsection{Higher residue pairing associated with Hertling's polarization $S$}
On the other hand, on $V^{>-1}=\bigoplus_{-1<\ka\leq 0}\mathbb{C}\{t\}C_\ka$  there is another higher residue pairing associated with the polarization $S$ in the polarized mixed Hodge structure by the following properties (see \cite{Hertling, He3} for details); letting  $-1<\alpha,\beta\leq 0$ and $a\in C_\alpha,b\in C_\beta$,
$$\begin{aligned}
  \sP:V^{>-1}\times V^{>-1}&\to\mathbb{C}\{\!\{\partial_{t}^{-1}\}\!\}\partial_{t}^{-1}\text{ where } \mathbb{C}\{\!\{\partial_{t}^{-1}\}\!\}:=\Big\{\sum_{i\geq0}a_i\partial_{t}^{-i}\mid \sum_{i\geq0}\frac{a_it^i}{i!}\in\mathbb{C}\{t\}\Big\}\\
  \sP(a,b)&:=\sum_{m\geq1}\sP^{(-m)}(a,b)\partial_{t}^{-m}\text{ where }\sP^{(-m)}(a,b)\in\mathbb{C};
\end{aligned}$$
\begin{enumerate}[label=\arabic*)]
  \item $\sP(a,b)=0$ if $\alpha+\beta\not\in\mathbb{Z}$;
  \item for  $\alpha+\beta=-1$ 
$$\sP(a,b)=\frac{1}{(2\pi i)^{n}}S(\psi^{-1}_{\alpha}(a),\psi^{-1}_{\beta}(b))\partial_{t}^{-1};$$
\item  for  $\alpha=\beta=0$,
  \begin{equation}\label{2_2_3}
    \sP(a,b)=\frac{-1}{(2\pi i)^{n+1}}S(\psi^{-1}_{\alpha}(a),\psi^{-1}_{\beta}(b))\partial_{t}^{-2};
  \end{equation}
\item for $g_1(\partial_{t}^{-1}), g_2(\partial_{t}^{-1})\in \mathbb{C}\{\!\{\partial_{t}^{-1}\}\!\}$,
  \begin{equation}\label{2_2_4}
    \sP(g_1(\partial_{t}^{-1})a,g_2(\partial_{t}^{-1})b)=g_1(\partial_{t}^{-1})g_2(-\partial_{t}^{-1})\sP(a,b).
    \end{equation}
\end{enumerate}

The following proposition would justify the term ``residue'' in the name of $\sP$.
\begin{proposition}\label{prop_1} \cite{BW2,Hertling,He3}
  \begin{align}
    \sP^{(-n-1)}(s_0(\omega_1),s_0(\omega_2))&=\sK^{(0)}(\omega_1,\omega_2)\\
    \sP^{(-m)}(s_0(\omega_1),s_0(\omega_2))&=0\text{ for }1\leq m\leq n.
    \end{align}
 \end{proposition}

\section{$G$-equivariant theory}\label{GGe}

\subsection{Milnor fibration in the $G$-equivariant setting}\label{milequi}
Let $G$ be a finite group acting linearly on $Q:=\mathbb{C}[x_0,\dots,x_n]$ as $\mathbb{C}$-algebra automorphisms of $Q$, $I_g$ be the ideal of $Q$ generated by $q - g q $ for all $q \in Q$.
Denote  $Q^g := Q / I_g$ and $f_g := f |_{Q^g}$ for $f\in Q$. Up to a change of variables, every action looks like the followings;
\begin{equation}\label{action1}
  g\cdot (x_0,\dots,x_n)=(l_0,\dots,l_{k_g},x_{k_g+1},\dots,x_{n})
  \end{equation}
where $l_0,\dots,l_{k_g}$ are linear forms in $x_0,\dots,x_{k_g}$ such that $l_0-x_0,\ldots,l_{k_g}-x_{k_g}$ are linearly independent. Thus $Q^g=Q/(x_0,\ldots,x_{k_g})$ and
$f_g(x_{k_g+1},\dots,x_{n})=f|_{x_0=\cdots=x_{k_g}=0}$.

Assume that $f\in Q=\mathbb{C}[x_0, \dots, x_n]$ is $G$-invariant and  the only singularity of the associated morphism $f: \Spec(Q) \to \A^1_\mathbb{C}$  of smooth affine varieties
is at $\fm=(x_0,\dots,x_n)$. Let $U^g=\Spec (Q^g)$ denote the $g$-fixed locus of $U=\Spec(Q)$ for $g\in G$ as smooth affine varieties. Note for $n_g+1:=\dim Q^g$, one has $U^g\cong\mathbb{C}^{n_g+1}$. We have a $G$-action on 
the disjoint union $\coprod _{g \in G} U^g$ by
\[ U^g \to U^{hgh^{-1}}, \ x \mapsto h x  \text{ for } h \in G .  \]
Let $X^g:=U^g\cap X$ where $X$ is as in \S~\ref{mil_fib}. Then
\[ X^g \to X^{hgh^{-1}}, \ x \mapsto h x  \text{ for } h \in G .  \]
$f_g: X^{g} \to T$ denotes the map induced by the restriction of $f:X\to T$ by abuse of notation.
Note that if necessary, shrinking the disk $T$ sufficiently, for all $g\in G$ we may assume that $f_g: X^{g} \to T$ is surjective  by the open mapping theorem unless $\dim_\mathbb{C}X^g=0$. We also let
$f_g' : (X^g)' \to T'$ where $(X^g)' := X^g\setminus f_{g}^{-1}(0)$ and $X^g_\infty:=(X^g)'\times_{T'}T_\infty$.

\subsection{$G$-equivariant matrix factorizations}\label{gmat}
Let $Q$ be $\mathbb{C}[x_0,\dots,x_n]$, $G$ be a finite group acting on $Q$ as automorphisms of $Q$, and $f\in Q$ be $G$-invariant. We define the $\mathbb{Z}/2$-graded dg-category $\MF_G(Q,f)$ of $G$-equivariant matrix factorizations as follows:
\begin{definition}
  A  $G$-equivariant matrix factorization of a potential $f$ over $Q$ is a 
pair
\begin{equation}
(E,\delta_E)=\xymatrix{(E^0 \ar@/^/[r]^{\delta_0} & E^1\ar@/^/[l]^{\delta_1})}\text{ where }
\end{equation}
\begin{itemize}
\item  $E=E^0\oplus E^1$ is a $\mathbb{Z}/2$-graded finitely generated projective $Q$-module equipped with
a compatible $G$-action, and
\item $\delta_E\in \End^1_Q(E)$ is an odd $G$-equivariant (i.e., of degree $1 \in \mathbb{Z}/2$)
endomorphism of $E$ such that $\delta_E^2=f\cdot \id_E$.  
\end{itemize}
 Morphisms between $G$-equivariant matrix factorizations  $\bar{E}:=(E,\delta_E)$ and $\bar{F}:=(F,\delta_F)$ should also be compatible with the action of $G$, so 
\begin{equation}\label{G-hom-eq}
  \Homb_{\MF_G(Q,f)}(\bar{E},\bar{F})=\Homb_{\MF(Q,f)}(\bar{E},\bar{F})^G
  \end{equation}
where $\MF(Q,f)$ is the category of matrix factorizations of a potential $f$ over $Q$ forgetting the $G$-equivariant structure.
\end{definition}
For instance, if $E^0$ and $E^1$ are free $G$-equivariant $Q^g$-modules with chosen bases over a fixed locus $U^g$, 
the differential $\delta_E$ can be represented by
a block matrix
\begin{align}
  \delta_E= \begin{pmatrix}0&D^1\\D^0&0\end{pmatrix}&\text{ such that }  D^1D^0= D^0D^1= f_g\cdot\mathbf{I}\text{ with}\label{nota3}\\
    \begin{pmatrix}0&D^1\\D^0&0\end{pmatrix} \begin{pmatrix}g^0&0\\0&g^1\end{pmatrix} &=
        \begin{pmatrix}g^0&0\\0&g^1\end{pmatrix} \begin{pmatrix}0&D^1\\D^0&0\end{pmatrix}\text{ where }g:=\begin{pmatrix}g^0&0\\0&g^1\end{pmatrix}.\label{equivariance}
\end{align}
In this case, we let
\begin{equation}\label{nota}
  (E,\delta_E):=(D^1,D^0).
  \end{equation}
\subsection{The Grothendieck group}\label{grogp}
Note that $K_0(\MF(Q,f)) \cong K_0^\Delta([\MF(Q,f)]^{\idem})$ where  $[\MF(Q,f)]$ is the the homotopy category of $\MF(Q,f)$ and $[\MF(Q,f)]^{\idem}$ is the idempotent completion; see \cite[Section 2.1]{BW2} for details. Moreover, if the only singularity of the associated morphism $f: \Spec(Q) \to \A^1_\mathbb{C}$  of smooth affine varieties
is at $\fm$, then $K_0(\MF(Q,f))\cong K_0^{\Delta}([\MF(\widehat{Q}, f)])$ where $\widehat{Q}$ denotes the $\fm$-adic completion of $Q$ (see \cite[Theorem 5.7]{dyckerhoff}). Furthermore, from \cite[Lemma 5.9]{BW2}, it is known that 
\begin{equation}\label{kgroup}
  K_0(\MF(Q, f)) \cong K_0^{\Delta}([\MF(Q^\hen,f)])
\end{equation}
where $Q^\hen$ is the henselization of $Q$ at $\fm$. From \eqref{kgroup}, we identify a class in  $K_0(\MF(Q, f))$ with the corresponding class in $K_0^{\Delta}([\MF(Q^\hen,f)])$. That is,  since $Q^\hen$ is a local ring, any class in $K_0(\MF(Q, f))$ may be represented by a free matrix factorization in $\MF(Q^\hen, f)$ and denoted by $(A,B)$ as in \eqref{nota}.

The action of $G$ on $Q$ is naturally extended to $R:=\mathbb{C}[\![x_0,\dots,x_n]\!]$ and its henselization $Q^\hen\subseteq \mathbb{C}[\![x_0,\dots,x_n]\!]$. $K_0(\MF_G(Q, f))$ is similarly given, i.e.,
  $$K_0(\MF_G(Q,f)) \cong K_0^\Delta([\MF_G(Q,f)]^{\idem}).$$
Since $K_0^\Delta([\MF_G(Q,f)]^{\idem})\cong K_0^{\Delta}([\MF_G(Q^\hen,f)])$ for a nontrivial finite group $G$; see \cite{BWidem} for details, in this paper we will merely work on  $K_0^{\Delta}([\MF_G(Q^\hen,f)])$, i.e., as in \eqref{nota} its class is represented by $(A,B)$ with entries in $Q^\hen$ and the equivariant relation \eqref{equivariance} is satisfied over $(Q^\hen)^g$ for each $g\in G$.

\subsection{Chern-character-type map $\ch_{X_\infty}$}\label{chentype}
For the Milnor fibration $f:X\to T$  as in \S~\ref{mil_fib} with $\dim_\mathbb{C}X=n+1$, assume that $n$ is odd and $p:=\frac{n+1}{2}$.
From \cite[Proposition 5.3]{BW2}, there is a well-defined Chern-character-type map $\ch_{X_\infty}:K_0(\MF(Q, f))\to H^{n}(X_\infty; \Q)_1$ defined by
\begin{equation}\label{chxh}
  \ch_{X_\infty}(A,B):=\left[\frac{-1}{(2 \pi i)^p}\frac{(p-1)!}{n!} \tr_{Q^\hen}(A^{-1} dA (dA^{-1} dA)^{p-1})\right]
  \end{equation}
for  $(A,B) \in K_0(\MF(Q,f))$.  From \cite[page 20 and 25]{BW2} we note that  $Q^\hen\subseteq\mathcal{O}_{X,0}^{\an}$ and choosing a lift of $(A,B) \in \MF(Q^\hen,f)$ to an object of $\MF(\widetilde{Q}, f)$ for some \'{e}tale extension $Q \subseteq \widetilde{Q}\subseteq Q^\hen$ and assuming $X$ is sufficiently small so that every element of $\widetilde{Q}$ converges absolutely on $X$ we may interpret $A$ and $B$ as matrices with entries in $\Gamma(X, \mathcal{O}_X^{an})$, i.e., $\tr_{Q^\hen}(A^{-1} dA (dA^{-1} dA)^{p-1})\in \Gamma(X', \Omega_X^{\an, n})$. Consequently, $\ch_{X_\infty}(A,B)$ is regarded as a class in $H^{n}(X_\infty; \Q)_1$ after composing the canonical map
\begin{equation}\label{labelDe}
  H^{n}_{deR}(X';\mathbb{C})\cong H^{n}(X'; \CC) \overset{\ccan}{\to} H^{n}(X_\infty; \CC)_1.
  \end{equation}
Additional details can be found in \cite[Section 4]{Hertling} and \cite[Section 5]{BW2}.

We  extend \eqref{chxh} to the equivariant setting. Assume $\dim_\mathbb{C}X^g:=n_g+1$ where $n_g$ is odd and $p_g:=\frac{n_g+1}{2}$. For a $G$-equivariant free matrix factorization $(A,B)$ over $Q^\hen$ and each $g\in G$, we may define
\begin{equation}\label{chxhh}
  \chxh(A,B):=\left[\frac{-1}{(2 \pi i)^{p_g}}\frac{(p_g-1)!}{n_g!} \tr_{(Q^\hen)^g}(A_g^{-1} dA_g (dA_g^{-1} dA_g)^{p_g-1}\circ g^1)\right]
  \end{equation}
where $A|_{(Q^\hen)^g}:=A_g$. Via the canonical restriction map from $K_0^{\Delta}([\MF_G(Q^\hen,f)])$ to each $K_0^{\Delta}([\MF((Q^\hen)^g,f_g)])$, we apply the same construction as in \cite[Section 5]{BW2} to each $K_0^{\Delta}([\MF((Q^\hen)^g,f_g)])$. Then we have a well-defined map
$$\chxh:K_0^{\Delta}([\MF_G(Q^\hen,f)])\to H^{n_g}(X^{g}_{\infty}; \mathbb{C})_1.$$ 
In passing, we observe that since the values of a character of a finite group are algebraic integers, the class $\chxh(A,B)$ is in $H^{n_g}(X^{g}_{\infty}; \mathbb{C})_1$. When $n_g$ is even or $\dim_\mathbb{C}X^g=0$, then we define $\chxh(A,B):=0$.

\subsection{$G$-equivariant Hochschild Chern character $\chhh$}\label{chench}
Let $f\in R=\mathbb{C}[\![x_0,\dots,x_n]\!]$ have an isolated singularity, i.e., $\dim_\mathbb{C}\mathbb{C}[\![x_0,\dots,x_n]\!]/(\partial_{x_0} f,\ldots,\partial_{x_n} f)<\infty$, and be invariant under a finite group $G$. Recalling the action of $g$ as in \eqref{action1}, we let
$$ H(f_g):=\frac{\Omega^{n_g+1}_{R^g/\CC}}{(-df_g)\wedge\Omega^{n_g}_{R^g/\CC}}\text{ for }R^g\cong \mathbb{C}[\![y_1,\dots,y_{n_g+1}]\!].$$
Then we have
\begin{equation}\label{G-hoch}
HH_*(\MF_G(R,f)) \cong (\bigoplus_{g\in G} H(f_g))^G\text{ (see \cite[Theorem~2.5.4]{PV: HRR})}
\end{equation}
where the right-hand side means the invariant part with respect to the action of $h\in G$ on $\bigoplus_{g\in G} H(f_g)$ by sending
$H(f_g)$ to $H(f_{hgh^{-1}})$. We remark that in \cite{BW,CKK} one can find the definition of a Hochschild homology $HH_\ast$ of a differential graded category and its identification in \eqref{G-hoch}, a so-called Hochschild-Kostant-Rosenberg type isomorphism. According to Polishchuk and Vaintrob \cite{PV: HRR}, the $G$-equivariant Chern character is characterized explicitly as follows:

\begin{proposition}\label{G-hoch-eq}{\em (\cite[Theorem~3.3.3]{PV: HRR}) }
For a $G$-equivariant matrix factorization $\bar{E}=(E,\delta_E)$ the $G$-equivariant Chern character  $\chhh(\bar{E})$ is given by
\begin{align}
  \chhh(\bar{E})&:=\bigoplus _{g \in G} \chhh(\bar{E})_g \in(\bigoplus_{g\in G} H(f_g))^G\label{G11}\text{ where }\\
\chhh(\bar{E})_g&:=\frac{(-1)^{n_g+1}\sTr_{R^g}((d\delta_E|_{R^g})^{n_g+1}\circ g)}{(n_g+1)!} \mod \JJ_{f_g}.\label{G-ch-eq}
\end{align}
Note that $\JJ_{f_g}$ is the Jacobi ideal generated by $df_g$ and $\sTr$ is a supertrace which is the difference of the traces on $E^0$ and $E^1$.
\end{proposition}

\subsection{$G$-equivariant negative cyclic Chern character $\chhn$}\label{chenhn}

Assume \ref{enun1} and \ref{enun2} in \S~\ref{mil_fib}. Then we have a generalization of \eqref{G-hoch} in the setting of \S~\ref{milequi} as follows: From \cite[Theorem 1.1]{Bar} (see also \cite[Theorem~6.5]{CKK}), we have
\begin{align}
  HN_*(\MF_G(Q,f)) \cong (\bigoplus_{g\in G} N(f_g))^G&\text{ where }\label{G-hocn1}\\
  N(f_g):=\frac{\Omega^{n_g+1}_{Q^g/\CC}[\![u]\!]}{(ud-df_g\wedge)\Omega^{n_g}_{Q^g/\CC}[\![u]\!]}&\text{ unless }Q^g=\CC\label{G-hocn2}
\end{align}
and the right-hand side means the invariant part with respect to the action of $h\in G$ on $\bigoplus_{g\in G} N(f_g)$ by sending
$N(f_g)$ to $N(f_{hgh^{-1}})$. We remark that in \cite{BW,CKK} one can find the definition of a negative cyclic homology $HN_\ast$ and its identification in \eqref{G-hocn1}. For a $G$-equivariant matrix factorization $\bar{E}=(E,\delta_E)$, after composing the Hochschild-Kostant-Rosenberg isomorphism \eqref{G-hocn1} the {\em categorical} $G$-equivariant negative cyclic homology valued Chern character $\Chhn(\bar{E})$ (see \cite{BW,CKK,shkl} for the definition)  becomes
$$\chhn(\bar{E}):=\bigoplus _{g \in G} \chhn(\bar{E})_g\in  (\bigoplus_{g\in G} N(f_g))^G.$$
In particular, if $n_g$ is even, then  $\chhn(\bar{E})_g=0$. Moreover, from \cite[Proposition 7.1]{BW2} (see also \cite{schul}) we have that
\begin{equation}\label{chn3}
  N(f_g) \cong\frac{\widehat{\Omega}^{n_g+1}_{g}[\![u]\!]}{(ud-df_g\wedge)\widehat{\Omega}^{n_g}_{g}[\![u]\!]} \cong\frac{\widehat{\Omega}^{n_g+1}_{g}}{df_g\wedge d\widehat{\Omega}^{n_g-1}_{g}}
\end{equation}
where $\widehat{\Omega}^\bullet_{g}:=\Omega^{\bullet}_{Q^g/\CC}\otimes_{Q^g}\widehat{Q^g}$ and $\widehat{Q^g}$ is the completion of $Q^g$ at $\fm$.

\subsection{Classifying stack $BG$ and character $\chg$}\label{chenclass}
Consider a classifying stack $BG := [\Spec\, \CC /G ]$ of a finite group $G$ over a point. One may identify $K_0(BG)$ with the space $\mathsf{R}(G)$ of finite dimensional complex linear representations of $G$  and define a character map 
\begin{equation}\label{ch for BG}
  \begin{aligned}
    \chg: K_0 (BG ) \ot \CC \cong \mathsf{R} (G) \otimes \CC &\to \bigoplus _{[g]\in G/G}  \CC \\
        [\rho] &\mapsto \bigoplus _{[g] \in G/G} \mathrm{tr}_\CC (\rho(g))
  \end{aligned}
\end{equation}
where $G/G$ denotes the set of all conjugacy classes of $G$. It is an isomorphism and gives a special case of \eqref{G11}.
We let $\chg(\rho)_g:=\mathrm{tr}_\CC(\rho(g))$.

\section{Proof of the main theorem}\label{mainproof}
Throughout this section,  assume \ref{enun1} and \ref{enun2} in \S~\ref{mil_fib} and that $n_g$ is odd. For a $G$-equivariant free matrix factorization $(A,B)$ over $Q^\hen$, we let
$$\chpv(A,B)_g:=\frac{2\tr_{(Q^\hen)^g}((dA_gdB_g)^{p_g}\circ  g^0)}{(n_g+1)!}\text{ where } p_g:=\frac{n_g+1}{2}.$$
We note that PV stands for Polishchuk and Vaintrob. After composing $\frac{\Omega^{\an, n_g+1}_{X^g,0}}{(-df_g) \wedge \Omega^{\an, n_g}_{X^g,0}}\to \frac{\widehat{\Omega}^{n_g+1}_{g}}{(-df_g)\wedge \widehat{\Omega}^{n_g}_{g}}$ with  $\frac{\widehat{\Omega}^{n_g+1}_{g}}{(-df_g)\wedge \widehat{\Omega}^{n_g}_{g}} \cong H(f_g):=\frac{\Omega^{n_g+1}_{Q^g/\CC}}{(-df_g)\wedge\Omega^{n_g}_{Q^g/\CC}}$, we have
\begin{equation}\label{brie2}
  [\chpv(A,B)_g]_{H(f_g)}=\chhh(A,B)_g\in H(f_g).
\end{equation}
We note that $\chhh(A,B)_g=0$ when $n_g$ is even. Similarly, the composition of  $\frac{\Omega^{\an, n_g+1}_{X^g,0}}{df_g \wedge d \Omega^{\an, n_g-1}_{X^g,0}}\to \frac{\widehat{\Omega}^{n_g+1}_{g}}{df_g\wedge d\widehat{\Omega}^{n_g-1}_{g}}$ with \eqref{chn3} transfers a Brieskorn class $[\chpv(A,B)_g]_B$ to a class
\begin{equation}\label{brie1}
[\chpv(A,B)_g]_{N(f_g)}=\chhn(A,B)_g\in N(f_g)\text{ (see \cite{BW,CKK})}.
\end{equation}
In particular, if  $[\chpv(A,B)_g]_B=0$, then $\chhn(A,B)_g=0$. Those two $G$-equivariant Chern characters are related by setting $u=0$. In other words,
\begin{equation}\label{pf1}
\chhn(A,B)_g\xmapsto{u=0}\chhh(A,B)_g\text{ (see \cite[page 194]{shkl} and \cite[Example 6.1]{BW}).}
\end{equation}
In particular, if  $\chhn(A,B)_g=0$, then $\chhh(A,B)_g=0$.

\begin{theorem}\label{th55}
Let $p_g:=\frac{n_g+1}{2}$. For
  $(A,B) \in \MF_G(Q^h, f)$, we have 
\begin{equation}
  \psi_0(\chxh(A,B)) =   \frac{1}{(2 \pi i)^{p_g}}\partial_{t}^{p_g-1}s_0
  \left([\chpv(A,B)_g]_B \right).
\end{equation}
\end{theorem}
\begin{proof}
A computation using that $A_gB_g = f_g$ and their differential relations shows that the class $\chxh(A,B)$ in $H^{n_g}_{deR}((X^g)'; \mathbb{C})$ in terms of \eqref{labelDe} is represented by

$$\frac{-(p_g-1)!}{(2 \pi i)^{p_g}n_g!} f_g^{-p_g}  \gtr\big(B_g dA_g (dB_g dA_g)^{p_g-1}\circ g^1\big) + df_g \wedge \omega$$
for some $\omega \in \Gamma((X^g)', \Omega^{\an, n_g-1}_{X^g})$. Since $df_g$ restricts to $0$ in the de Rham cohomology of $X_t^g:=f_g^{-1}(t)$ for $t\ne0$,
$\psi_0(\chxh(A,B)) \in \cG_0^g:=i_\ast(\mathbf{R}^{n_g} {f_g'}_{\ast} \mathbb{C}_{(X^g)'}  \otimes_{\mathbb{C}_{T'}} \cO^{an}_{T'})_0$ is given by the germ
$$\left(t \mapsto \left[\frac{-(p_g-1)!}{(2 \pi i)^{p_g}n_g!}t^{-p_g}  \gtr\big(B_g dA_g (dB_g dA_g)^{p_g-1}\circ g^1\big)|_{X^g_t}\right]\in
H^{n_g}(X_t^g; \mathbb{C})\right)_0.$$
Observe that using $dB_{g}^{-1}=-B_{g}^{-1}(dB_{g})B_{g}^{-1}$,  $ \tr(ab)=(-1)^{|a||b|} \tr(ba)$, and the equivariant relation \eqref{equivariance} in order, we have
$$\begin{aligned}
   \gtr\big((dB_g^{-1} dB_g)^{p_g}\circ  g^0\big) &= (-1)^{p_g}  \gtr\big((B_g^{-1} dB_g)^{2p_g}\circ  g^0\big)\\
  &=-(-1)^{p_g}  \gtr\big(B_g^{-1} dB_g\circ  g^0\circ (B_g^{-1} dB_g)^{2p_g-1}\big)\\
   &=-(-1)^{p_g}  \gtr\big((B_g^{-1} dB_g)^{2p_g}\circ  g^0\big)\\
   &= -\gtr\big((dB_g^{-1} dB_g)^{p_g}\circ  g^0\big),
\end{aligned}$$
which implies that
\begin{equation}\label{42}
 \gtr\big((dB_g^{-1} dB_g)^{p_g}\circ  g^0\big)=0.
\end{equation}
Thus, by $df_g\wedge df_g=0$, \eqref{42}, and $A_gB_g = f_g$, we see that
\begin{equation}\label{43}
  \begin{aligned}
  & \gtr\big((dA_g dB_g)^{p_g}\circ g^0\big)=  \gtr\big(( df_g B_g^{-1} dB_g+f_g dB_g^{-1} dB_g)^{p_g}\circ g^0\big) &\\
& =   f_g^{p_g}  \gtr\big((dB_g^{-1}dB_g )^{p_g}\circ g^0\big)+ p_g f_g^{p_g-1} df_g  \gtr\big(B_g^{-1} dB_g (dB_g^{-1} dB_g)^{p_g-1}\circ g^0\big) \\
&= p_g f_g^{-1} df_g  \gtr\big(A_gdB_g (dA_g dB_g)^{p_g-1}\circ g^0\big).
  \end{aligned}
  \end{equation}
By applying \eqref{43}, the property of the trace map, \eqref{equivariance}, and the differential relations from $B_gA_g=f_g$ in order, we have
$$\begin{aligned}
  &\frac{1}{(2 \pi i)^{p_g}}\partial_t^{p_g-1}s_0 \left([\chpv(A,B)_g]_B\right)&\\
  &= \left(t\mapsto \left[\frac{2p_g(p_g-1)!(-1)^{p_g-1}}{(2 \pi i)^{p_g}(n_g+1)!}t^{-p_g} \gtr\big(A_gdB_g (dA_g dB_g)^{p_g-1}\circ g^0\big)|_{X^g_t}\right]\in
  H^{n_g}(X_t^g; \mathbb{C})\right)_0&\\
 &=\left(t\mapsto \left[\frac{(p_g-1)!}{(2 \pi i)^{p_g}n_g!}t^{-p_g} \gtr\big(dB_g\circ g^0\circ A_g (dB_gdA_g)^{p_g-1}\big)|_{X^g_t}\right]\in
  H^{n_g}(X_t^g; \mathbb{C})\right)_0&\\
   &=\left(t\mapsto \left[\frac{(p_g-1)!}{(2 \pi i)^{p_g}n_g!}t^{-p_g} \gtr\big(dB_gA_g (dB_gdA_g)^{p_g-1}\circ g^1\big)|_{X^g_t}\right]\in
H^{n_g}(X_t^g; \mathbb{C})\right)_0&\\
  &=\left(t\mapsto \left[\frac{-(p_g-1)!}{(2 \pi i)^{p_g}n_g!}t^{-p_g} \gtr\big(B_gdA_g (dB_gdA_g)^{p_g-1}\circ g^1\big)|_{X^g_t}\right]\in
H^{n_g}(X_t^g; \mathbb{C})\right)_0&\\
  &=\psi_0(\chxh(A,B)).& 
\end{aligned}$$

\end{proof}

\begin{corollary} \label{cor913}
For any $(A,B)\in K_0^{\Delta}([\MF_G(Q^\hen,f)])$ and $g\in G$, we have 
$$
S(\chxh(A,B), \chxh(A,B)) \geq 0.
$$ 
In particular,
$$S(\chxh(A,B), \chxh(A,B)) = 0 \text{ if and only if }
\chxh(A,B) = 0$$
where $S$ is the polarization on $H^{n_g}(X_\infty^g; \mathbb{Q})_1$.
  \end{corollary}
\begin{proof}
$\chxh(A,B)$ is contained in $\IIm(H^{n_g}((X^g)'; \mathbb{C}) \overset{\ccan}{\to} H^{n_g}(X^g_\infty; \CC)_1) = \ker(N)$ by construction. Thus, Theorem~\ref{th55} and the definition of the decreasing filtration $F^\bullet$ in \eqref{pmhsst} imply
$$\chxh(A,B) \in \ker(N) \cap F^{p_g} H^{n_g}(X_\infty^g; \mathbb{C})_1.$$
  Since the polarization $S$ is positive definite on $\ker(N) \cap F^{p_g} H^{n_g}(X_\infty^g; \mathbb{C})_1$ from Remark~\ref{remark1}, the proof is established.
\end{proof}

\begin{lemma}\label{BvSS}
For any free matrix factorizations $(A',B'), (A,B) \in \MF_G(Q^\hen, f)$ and $g\in G$, we have
$$\Gres(\chpv(A',B')_{g},\chpv(A,B)_g)=(-1)^{p_g}S(\chxh(A',B'),\chxh(A,B)).$$
\end{lemma}
\begin{proof}
 We see that
  \begin{equation}\label{eq11}
    \begin{aligned}
      &\sP\left(s_0 ([\chpv(A',B')_g]_B),s_0 ([\chpv(A,B)_g]_B)\right)&\\
      &=(2\pi i)^{n_g+1}\sP\left(\partial_{t}^{-p_g+1}\psi_0(\chxh(A',B')),\partial_{t}^{-p_g+1}\psi_0(\chxh(A,B))\right)\quad\text{by Theorem~\ref{th55}}&\\
      &=(-1)^{p_g-1}(2\pi i)^{n_g+1}\sP\left(\psi_0(\chxh(A',B')),\psi_0(\chxh(A,B))\right)\partial_{t}^{-n_g+1}\text{by \eqref{2_2_4}}&\\
      &=(-1)^{p_g}S(\chxh(A',B'),\chxh(A,B))\partial_{t}^{-n_g-1}\qquad\qquad\qquad\qquad\qquad\:\,\,\,\text{by \eqref{2_2_3}.}&
    \end{aligned}
\end{equation}
On the other hand, by Proposition~\ref{prop_1} we have
  \begin{multline}\label{eq12}
    \sP\left(s_0 ([\chpv(A',B')_g]_B),s_0 ([\chpv(A,B)_g]_B)\right)\\
    =\sK^{(0)}([\chpv(A',B')_{g}]_B,[\chpv(A,B)_g]_B)\partial_{t}^{-n_g-1}\\
    +\sum_{m\geq n_g+2} \sP^{(-m)}\left(s_0 ([\chpv(A',B')_g]_B),s_0 ([\chpv(A,B)_g]_B)\right)\partial_{t}^{-m}.
    \end{multline}
By comparing \eqref{eq11} with \eqref{eq12}, we conclude that
\begin{multline*}
  \sP^{(-m)}\left(s_0 ([\chpv(A',B')_g]_B),s_0 ([\chpv(A,B)_g]_B) \right)=0 \text{ for $m\geq n_g+2$ and }\\
   \sK^{(0)}([\chpv(A',B')_{g}]_B,[\chpv(A,B)_g]_B)=(-1)^{p_g}S(\chxh(A',B'),\chxh(A,B)).
\end{multline*}
The definition of $\sK^{(0)}$ on a Brieskorn lattice $\Omega^{\an, n_g+1}_{X^g,0}/df_g \wedge d \Omega^{\an, n_g-1}_{X^g,0}$ in \S~\ref{Saitono} establishes the proof.
\end{proof}

We observe that the dual matrix factorization  $(E^\ast,\delta_{E}^{\ast}):=(A,B)^\ast$ of $(E,\delta_E)=(A,B)\in\MF_G(Q^\hen,f)$ is an object in $\MF_G(Q^\hen,-f)$ (see \cite[Section 2.2]{PV: HRR}) and $(A,B)^\ast=(-B^\ast,A^\ast)$. We remark  that 
 \begin{equation}\label{dual1}
    \chpv((A,B)^\ast)_{g^{-1}}=\chpv(A,B)_{g}\text{ (see \cite[Lemma 3.3.5]{PV: HRR}).}
    \end{equation}
A similar proof gives the following.

\begin{lemma}\label{lemma2}
  Let $(A,B) \in \MF_G(Q^\hen, f)$.
$$\chpv((A,-B)^\ast)_{g^{-1}}=(-1)^{{n_g + 1 \choose 2}}\chpv(A,B)_{g}.$$
  \end{lemma}

\begin{proof}
Let $(A,B)\in\MF_G(Q^\hen,f)$. Then $(A,-B)^\ast=(B^\ast,A^\ast)\in  \MF_G(Q^\hen, f)$. By $\tr(a^\ast)= \tr(a)$ and the equivariant relation \eqref{equivariance}, we have
  \begin{align*}
  &\chpv((A,-B)^\ast)_{g^{-1}}=\frac{2 \gtr\left((d(B_g)^\ast d(A_g)^\ast)^{p_g}\circ (g^0)^\ast\right)}{(n_g+1)!}&\\
    =& (-1)^{{n_g + 1 \choose 2}}\frac{2 \gtr\left( \left(g^0\circ(dA_gdB_g)^{p_g}\right)^\ast\right)}{(n_g+1)!}=(-1)^{{n_g + 1 \choose 2}}\frac{2 \gtr\left( g^0\circ(dA_gdB_g)^{p_g}\right)}{(n_g+1)!}&\\
    =& (-1)^{{n_g + 1 \choose 2}}\frac{2 \gtr\left( (dA_gdB_g)^{p_g}\circ g^0)\right)}{(n_g+1)!}= (-1)^{{n_g + 1 \choose 2}}\chpv(A,B)_{g}.&
  \end{align*}
\end{proof}

 \begin{Rmk}\label{lem45}
  Under \ref{enun1} and \ref{enun2} in \S~\ref{mil_fib}, from \cite[Theorem 1.1]{Bar} and the Hochschild-Kostant-Rosenberg isomorphism again we have
\begin{align}
  HP_*(\MF_G(Q,f)) \cong (\bigoplus_{g\in G} P(f_g))^G&\text{ where }\label{G-hopn1}\\
  P(f_g):=\frac{\Omega^{n_g+1}_{Q^g/\CC}(\!(u)\!)}{(ud-df_g\wedge)\Omega^{n_g}_{Q^g/\CC}(\!(u)\!)}&\text{ unless }Q^g=\CC.\label{G-hopn2}
\end{align}
In \cite{BW,CKK} one can find the definition of a periodic cyclic homology $HP_\ast$ and its identification in \eqref{G-hopn1}. The naturality of Chern character maps gives the following commutative diagram;
$$\xymatrix{ K_0 (\MF_G(Q,f)) \ar[rr]^{\Chhn} \ar[rd]_{\Chhp} &   & \ar[dl]_{\ccan} HN_* ( \MF_G(Q,f))    &   \\
         & HP_* ( \MF_G(Q,f))  &  
  &      }$$
where $\Chhp$ is the $G$-equivariant periodic cyclic homology valued Chern character. Since the canonical map $\ccan$ becomes an embedding under \ref{enun2} in \S~\ref{mil_fib} (see \cite[page 194]{shkl} and \cite[page 32]{BW2}), via the Hochschild-Kostant-Rosenberg isomorphisms which send categorical Chern characters $\Chhn$ and $\Chhp$ to twisted differential form valued Chern characters $\chhn$ and $\chhp$ respectively; see \cite{BW,CKK,shkl},  we may see that $\chhn(\bar{E})=0$ if and only if $\chhp(\bar{E})=0$ for $\bar{E} \in K_0(\MF_G(Q, f))$.
\end{Rmk}

\begin{theorem}\label{math1}
  Let $(A,B) \in K_0(\MF_G(Q,f))$. For all $G$-equivariant matrix factorizations $(A',B') \in K_0(\MF_G(Q,f))$, the Euler pairing
  $$\chi\big((A',B'), (A,B)\big)= 0$$
  if and only if $\chhp(A,B) = 0$.
  \end{theorem}

\renewcommand\qedsymbol{\textbf{Q.E.D.}}

\begin{proof}
Recalling $K_0(\MF_G(Q,f))\cong K_0^{\Delta}([\MF(Q^\hen,f)])$ in Section~\ref{grogp}, it suffices to prove this theorem for $ K_0^{\Delta}([\MF(Q^\hen,f)])$. Let $(A,B) \in  K_0^{\Delta}([\MF_G(Q^\hen,f)])$, $g \in G$, and $n_g$ be odd. By \eqref{ch for BG} there exists a representation $\eta_g \in K_0 (BG)$ such that $\chg (\eta_g) _h = 0$ for $h \notin [g^{-1}]$ and $\chg (\eta_g ) _{g^{-1}} \ne0$. Take $(A',B') := \eta_g \ot (A,-B)^\ast \in  K_0^{\Delta}([\MF_G(Q^\hen,f)])$. Recalling the Hirzebruch–Riemann–Roch theorem for $G$-equivariant matrix factorizations \cite[Theorem 4.2.1]{PV: HRR}
  \begin{equation}\label{HRrR}
    \chi\big((A',B'), (A,B)\big)  
  =        \sum_{h\in G}\frac{(-1)^{{n_h + 1 \choose 2}}\rrGres(\chhh(A',B')_{h^{-1}},\chhh(A,B)_{h})}{|G|\det(\id-h : U/U^h)},
\end{equation}
from \eqref{brie2} and \eqref{HRrR} we obtain
\begin{align*}
  &\chi\big((A',B'), (A,B)\big)  & \\
  = &       \sum_{h\in G}\frac{(-1)^{{n_h + 1 \choose 2}}\rrGres([\chpv(A',B')_{h^{-1}}]_{H(f_{h^{-1}})},[\chpv(A,B)_{h}]_{H(f_h)})}{|G|\det(\id-h : U/U^h)}    &   \\
 = &    \frac{(-1)^{{n_g + 1 \choose 2}}\chg(\eta_g)_{g^{-1}}\mathsf{c}_g}{|G|\det(\id-g : U/U^g)} \Gres([\chpv((A,-B)^\ast)_{g^{-1}}]_{H(f_{g^{-1}})},
  [\chpv(A,B)_{g}]_{H(f_g)})   &    \\
  = &    \frac{ (-1)^{{n_g + 1 \choose 2}}\chg(\eta_g)_{g^{-1}}\mathsf{c}_g}{|G|\det(\id-g : U/U^g)} \Gres(\chpv((A,-B)^\ast)_{g^{-1}},
    \chpv(A,B)_{g})   &    \\
    = &     \frac{ \chg(\eta_g)_{g^{-1}}\mathsf{c}_g}{|G|\det(\id-g : U/U^g)}\Gres(\chpv(A,B)_g,\chpv(A,B)_{g}) \qquad\quad\;\;\;\;\;\text{ by Lemma~\ref{lemma2}}  &   \\
= & \frac{ (-1)^{p_g}\chg(\eta_g)_{g^{-1}}\mathsf{c}_g}{|G|\det(\id-g : U/U^g)} S(\chxh(A,B),\chxh(A,B)) \qquad\qquad\qquad\quad\text{ by Lemma~\ref{BvSS}. }&
\end{align*}
where $\mathsf{c}_{g}:=[G:C_G(g)]$. Thus,  $\chi\big((A',B'), (A,B)\big)= 0$ implies
$$S(\chxh(A,B),\chxh(A,B))=0.$$
By Corollary~\ref{cor913}, we deduce that $\chxh(A,B)=0$. Combining it with Theorem~\ref{th55} and recalling the injectivity of $\psi_0,s_0,\partial_{t}^{-1}$, $[\chpv(A,B)_g]_B=0$ in $\frac{\Omega^{\an, n_g+1}_{X^g,0}}{df_g \wedge d \Omega^{\an, n_g-1}_{X^g,0}}$. By \eqref{brie1}, we establish $\chhn(A,B)_g=0$. Since it is true for all $g\in G$ with odd $n_g$ and $\chhn(A,B)_g=0$ when $n_g$ is even, we conclude that $\chhn(A,B)=0$ which implies that $\chhp(A,B)=0$ by Remark~\ref{lem45}. 

On the other hand, combining Remark~\ref{lem45} with \eqref{pf1} one can see that the converse also follows easily from the Hirzebruch–Riemann–Roch theorem for $G$-equivariant matrix factorizations.

\end{proof}


\begin{thebibliography}{99}


\bibitem{AGZV}
V.I. Arnold, A.N. Varchenko, and S.M. Gusein-Zade, Singularities of
  differentiable maps volume {II}: monodromy and asymptotic integrals,
  vol.~83, Birkh\"auser, 1988.

\bibitem{Bar}
  V. Baranovsky, Orbifold cohomology as periodic cyclic homology, Internat. J. Math. \textbf{14} (2003), no. 8, 791-812.


\bibitem{BW}
  M. Brown and M. Walker, 
A Chern-Weil formula for the Chern character of a perfect curved module, J. Noncommut. Geom. \textbf{14} (2020), no. 2, 709-772.
  
\bibitem{BW2}
  \bysame, 
  Standard conjecture D for matrix factorizations, Advances in
  Mathematics \textbf{366} (2020), 40pp.

\bibitem{BWidem}
  \bysame,
  Idempotent completions of equivariant matrix factorization categories, arXiv:2212.14469v1.


\bibitem{CKK}
  K. Chung, B. Kim, and T. Kim, A Chain-level HKR-type map and a Chern character formula, arxiv:2109.14372.


\bibitem{dyckerhoff}
T. Dyckerhoff, Compact generators in categories of matrix
  factorizations, Duke Mathematical Journal \textbf{159} (2011), no.~2,
223--274.


\bibitem{fulton}
W. Fulton, Intersection theory, Ergeb. Math. Grenzgeb. (3), vol. 2, Springer-Verlag, Berlin, 1998. 

\bibitem{Hertling}
C. Hertling, Classifying spaces for polarized mixed {H}odge structures
  and for {B}rieskorn lattices, Compositio Mathematica \textbf{116} (1999),
  no.~1, 1--37.

\bibitem{Hertlingbook}
\bysame, Frobenius manifolds and moduli spaces for singularities, vol.
  151, Cambridge University Press, 2002.

\bibitem{He3}
  \bysame, Formes bilin\'eaires et hermitiennes pour des 
   singularit\'es: un aper\c{c}u. In: Singularit\'es (ed. D. Barlet), Institut \'Elie Cartan Nancy {\bf 18}, 2006, 1--17. 


\bibitem{Kulikov}
V.S. Kulikov, Mixed {H}odge structures and singularities, vol. 132, Cambridge University Press, 1998.


\bibitem{MT}
M. Marcolli and G. Tabuada, Noncommutative numerical motives, {T}annakian structures, and
  motivic {G}alois groups, Journal of the European Mathematical Society
  \textbf{18} (2016), no.~3, 623--655.

\bibitem{PV: HRR}
A. Polishchuk and A. Vaintrob,  
Chern characters and Hirzebruch-Riemann-Roch formula for matrix factorizations, Duke Mathematical Journal \textbf{161} (2012), no. 10, 1863-1926. 


\bibitem{SS}
J.~Scherk and J.H.M. Steenbrink, On the mixed {H}odge structure on the
  cohomology of the {M}ilnor fibre, Mathematische Annalen \textbf{271} (1985),
  no.~4, 641--665.

\bibitem{schul}
M. Schulze, A normal form algorithm for the {B}rieskorn lattice,
  Journal of Symbolic Computation \textbf{38} (2004), no.~4, 1207--1225.

\bibitem{shkl}
D. Shklyarov, Matrix factorizations and higher residue pairings, Advances in
  Mathematics \textbf{292} (2016), 181--209.

\bibitem{steenbrink}
J.H.M. Steenbrink, Mixed {H}odge structure on the vanishing cohomology,
  Real and complex singularities ({P}roc. 9th {N}ordic summer school {NAVF}
  {S}ympos. {M}ath., {O}slo 1976), 1977, pp.~565--678.


\bibitem{tabuada2}
G. Tabuada, A note on Grothendieck’s standard conjectures of type $\mathrm{C}^+$
and $\mathrm{D}$ in positive characteristic, Proc. Am. Math. Soc. \textbf{147}, No. 12, 5039-5054 (2019). 

\end{thebibliography}
\end{document}